\documentclass{amsart}
\usepackage[utf8]{inputenc}

\usepackage{amsmath}
\usepackage{amssymb}
\usepackage{verbatim}
\usepackage{geometry}
\usepackage{tikz}
\usepackage{soul}

\usepackage[colorlinks=true,urlcolor=blue,citecolor=red,linkcolor=blue,linktocpage,pdfpagelabels,bookmarksnumbered,bookmarksopen]{hyperref} 

\usepackage[hyperpageref]{backref}

\date{\today}

\newtheorem{thm}{Theorem}[section]
\newtheorem{cor}[thm]{Corollary}

\newtheorem{lemma}[thm]{Lemma}
\theoremstyle{definition}
\newtheorem{defn}[thm]{Definition}

\newtheorem{rem}[thm]{Remark}
\newtheorem*{ack}{Acknowledgments}

\numberwithin{equation}{section}

\title[]{On the sharp Makai inequality}

\author[Prinari]{Francesca Prinari}

\address[F. Prinari]{Dipartimento di Scienze Agrarie, Alimentari e Agro-ambientali
	\newline\indent 
	Universit\`a di Pisa
	\newline\indent
	Via del Borghetto 80, 56124 Pisa, Italy}
\email{francesca.prinari@unipi.it}

\author[Zagati]{Anna Chiara Zagati}
\address[A.\,C.\ Zagati]{Dipartimento di Scienze Matematiche, Fisiche e Informatiche
	\newline\indent
	Universit\`a di Parma
	\newline\indent
	Parco Area delle Scienze 53/a, Campus, 43124 Parma, Italy}
\email{annachiara.zagati@unipr.it}

\dedicatory{Dedicated to Prof. Giuseppe Buttazzo}

\begin{document}
	
\begin{abstract}
On a convex bounded open set, we prove that Poincaré--Sobolev constants for functions vanishing at the boundary can be bounded from below in terms of the norm of the distance function in a suitable Lebesgue space.
This generalizes a result shown, in the planar case, by E. Makai, for the torsional rigidity. In addition, we compare the sharp Makai constants obtained in the class of convex sets with the optimal constants defined in other  classes of open sets.  
Finally, an alternative proof of the {\it Hersch--Protter inequality} for convex sets is given.
\end{abstract}

\maketitle

\begin{center}
	\begin{minipage}{11cm}
		\small
		\tableofcontents
	\end{minipage}
\end{center}

\vskip.1cm

	
\section{Introduction}
The aim of this paper is to provide  a sharp lower bound for the quantity 
\begin{equation}\label{prodotto} \lambda_{p,q}(\Omega)\,\|d_{\Omega}\|_{L^{\frac{p\,q}{p-q}}(\Omega)}^p, 
\end{equation}
where  $\Omega\subsetneq\mathbb{R}^N $ is an convex bounded open set,   $1\le q<p<\infty$ or   $1<q=p<\infty,$ $\lambda_{p,q}(\Omega)$ is the {\it generalized principal frequency} defined as 
\[ \lambda_{p,q}(\Omega) :=\inf_{\psi \in C^{\infty}_0(\Omega)} \left\{\int_{\Omega} |\nabla \psi|^p \, dx\, :\, \int_{\Omega} |\psi|^q \, dx=1\right\},\] 
and $d_{\Omega}$ is the  distance function  from the boundary $\partial\Omega$, namely 
\[d_\Omega(x):=\inf\big\{|x-y| : y\in \partial\Omega \big\}.\]

Here and in what follows,  $L^\frac{p\,q}{p-q}(\Omega)$ stands for $L^\infty(\Omega)$ when $p=q$ and we will  write $\lambda_{p}(\Omega)$ in place of $\lambda_{p,p}(\Omega)$.

\vskip.2cm
\noindent Our study is motivated  by an old result  due to  Makai (see \cite{Makai}) for the torsional rigidity \[ T(\Omega)=\frac{1}{\lambda_{2,1}(\Omega)},\]
which asserts that, for every convex bounded open set $\Omega \subsetneq \mathbb{R}^2$, the following sharp upper bound holds
\[T(\Omega) \le \displaystyle \int_{\Omega} d_{\Omega}^{\,2}\, dx.\]



\subsection{Optimal lower bound on convex sets: the main result}  
Inspired by the above  {\it Makai inequality}, in this paper we will prove the following theorem which extends, to every dimension $N$ and every $1\le q<p<\infty$, the optimal lower bound given by Makai in the case $q=1$ and $p=2$.

\begin{thm}[Makai's inequality]\label{thm:makai}
Let $1 \le  q < p < \infty$ and let $\Omega \subsetneq \mathbb{R}^N$ be a convex bounded open set. Then, the following lower bound holds
	\begin{equation}\label{eq:makai2dim}
		\lambda_{p,q}(\Omega) \ge \frac{C_{p,q}}{\left(\displaystyle \int_{\Omega} d_{\Omega}^{\frac{p\,q}{p-q}}\, dx \right)^{\frac{p-q}{q}}}  ,
	\end{equation}
where $C_{p,q}$ is the positive constant given by 
	\begin{equation}\label{eq:cpq}
		C_{p,q}=
		\left(\displaystyle \frac{\pi_{p,q}}{2} \right)^p \left(\displaystyle  \frac{p-q}{pq+p-q} \right)^{\frac{p-q}{q}},
\end{equation}
	with 
	\begin{equation}\label{eq:pipq}
	\pi_{p,q} := \inf_{u \in C_0^{\infty}((0,1))} \left\{ \|u'\|_{L^p([0,1])}\, :\, \|u\|_{L^q([0,1])}=1 \right\}. 
\end{equation}
	Moreover, the estimate \eqref{eq:makai2dim} is sharp.
\end{thm}
The proof of Theorem \ref{thm:makai} is inspired by the covering argument for polygonal sets exploited by Makai in the planar case. In the $N$-dimensional case, thanks to a standard approximation argument, we can restrict ourselves to consider the case when $\Omega\subsetneq \mathbb{R}^N$ is the interior of a polytope $K$ (see Section \ref{sec:makai}). In this case, in order to prove \eqref{eq:makai2dim}, the key tool we use is given by Lemma \ref{partizione} where we  construct a suitable covering of $\Omega$ by means of convex sets $\Omega_i$, every one satisfying the property that $\partial \Omega_i\cap \Omega$ is the graph of a continuous function defined on a facet $S_i$ of the polytope $K$. The proof of the convexity of each set $\Omega_i$ relies on the concavity property of the distance function $d_{\Omega}$ (see \cite{AU}).

\vskip.2cm

If we denote by $r_{\Omega}$ the {\it inradius} of $\Omega$, which coincides with supremum of the distance function $d_{\Omega}$, 
as an application of Theorem \ref{thm:makai}, we can  give a different proof of  the following sharp estimate \eqref{eq:hpinequality},  first  proved  in \cite[Theorem 1.1]{BM} when  $1\le q<2$ and then  extended  to cover  the case $p \ne 2$ and $q=1$ in \cite[Theorem 4.3]{DPGG}. The general case   $1\le q<p<\infty$ was first  shown  in \cite[Theorem 5.7]{BPZ1} by means of a  comparison argument.   
\begin{cor}[Hersch-Protter--type inequality]
	\label{cor:HP}
	Let $1 \le q < p<\infty$ and $\Omega \subsetneq \mathbb{R}^N$ be a convex bounded open set. Then, the following lower bound holds
	\begin{equation}\label{eq:hpinequality}
		\lambda_{p,q}(\Omega)\,|\Omega|^{\frac{p-q}{q}} \ge \left( \frac{\pi_{p,q}}{2} \right)^p\,\frac{1}{r_\Omega^{\,p}}.
\end{equation}
	Moreover, the estimate \eqref{eq:hpinequality} is sharp.
\end{cor}

 A  further result that follows from  \eqref{eq:hpinequality}, simply taking the limit as $q\nearrow p$\footnote{We use here that $q\mapsto \lambda_{p,q}(\Omega)$ is left-continuous at $q=p$ when $\Omega\subset\mathbb{R}^N$ is a bounded open set.}, is the  sharp inequality 
\begin{equation}\label{eq:hpinequality1}
\lambda_{p}(\Omega) \ge \left( \frac{\pi_p}{2} \right)^p \frac{1}{r_{\Omega}^{\,p}}, 
\end{equation}
valid  for every convex bounded open set $\Omega \subsetneq \mathbb{R}^N$. Here $\pi_{p}$ is defined as in \eqref{eq:pipq} with $p=q$.
\noindent  Formula \eqref{eq:hpinequality1} represents the extension to the case of the $p$-Laplacian of the {\it Hersch--Protter inequality}, originally proved by Hersch  in \cite[Th\'eor\`eme 8.1]{He} for $p=N=2$ and then generalized to every dimension $N \ge 2$ by Protter in \cite[Theorem 2]{Pr} (see also \cite{BGM}). The case $p\ne 2$ has been already obtained in \cite{Bra2, DDG, Ka1}. In Section \ref{sec:HP}, we will give a further alternative proof of \eqref{eq:hpinequality1} by exploiting a change of variables formula  \eqref{thm:crastamalusa}, proved in \cite[Theorem 7.1]{CM} when the domain of integration is a connected open set of class $C^2$, and then using a suitable approximation result for convex sets.

\vskip.2cm


\subsection{Lower bounds on other classes of open sets}\label{sec:constants} The final part of this paper is devoted to compare the optimal constant $C_{p,q}$ for convex sets, defined by \eqref{eq:cpq}, with the infimum of  \eqref{prodotto} in other classes of open sets of $\mathbb{R}^N$. First of all, for every $ 1\le q< p<\infty$ or $1<q=p<\infty$, we introduce  the constant   
	\begin{equation}\label{eq:makaiaperti}
		\widetilde{C}_{p,q}=\inf \left\{ \lambda_{p,q}(\Omega)\,\|d_{\Omega}\|_{L^{\frac{p\,q}{p-q}}(\Omega)}^p: \, \Omega\subsetneq \mathbb{R}^N \mbox{ is an open set},\, d_{\Omega}\in L^{\frac{p\,q}{p-q}}(\Omega)\right\}.
			\end{equation}
			The condition on $d_{\Omega}$ is motivated by the recent results in \cite{BPZ2} where,  by means of a comparison principle provided in \cite{BPZ1}  for the sub-homogeneous Lane--Emden equation, 
  it is shown  that, when $\Omega \subsetneq \mathbb{R}^N$ is an open set and  $1\le q<p<\infty$ or $1<p=q<\infty$, then the following implication holds
   \[\lambda_{p,q}(\Omega)>0 \qquad  \Longrightarrow \qquad d_{\Omega}\in L^{\frac{p\,q}{p-q}}(\Omega)\]
(see \cite[Theorems 5.1\, and 5.4]{BPZ2}). Moreover, in the same paper, following an argument used  in \cite{vBe} when $p=2$,    the above  implication is shown to be  an equivalence in the class of the  open sets $\Omega\subsetneq \mathbb{R}^N$ which satisfy the   Hardy inequality 
\begin{equation}\label{eq:hardy}   
\int_{\Omega} |\nabla u|^p \, dx \ge C \, \int_{\Omega} \frac{|u|^p}{d_{\Omega}^{\,p}} \, dx, \qquad \mbox{ for every } u \in C^{\infty}_0(\Omega), 
\end{equation}
with   $1<p<\infty$ and $C=C(p,\Omega)>0$.

\noindent Indeed, if   \begin{equation}\label{summdist} d_{\Omega}\in L^{\frac{p\,q}{p-q}}(\Omega) \qquad \hbox{ for }  1\le q<p<\infty \hbox{ or } 1< q=p<\infty, 
\end{equation}
then, the joint application of the H\"older inequality and of \eqref{eq:hardy}, gives 
\[\begin{split}
\int_{\Omega} |u|^q \, dx &\le \left( \int_{\Omega} \frac{|u|^p}{d_{\Omega}^{\,p} } \, dx \right)^\frac{q}{p}\, \|d_{\Omega}\|_{L^{\frac{p\,q}{p-q}}(\Omega)}^p  \\
&\le \frac{1}{C^{\,\frac{q}{p}}}   \left(\int_{\Omega} |\nabla u|^p \, dx \right)^{\frac{q}{p}}\,\|d_{\Omega}\|_{L^{\frac{p\,q}{p-q}}(\Omega)}^p, \qquad  \mbox{ for every } u\in C^{\infty}_0(\Omega), 
\end{split}
\]
 that implies   \begin{equation}\label{eq:hardy2}  
	\lambda_{p,q}(\Omega)\ge \dfrac{\mathfrak{h}_p(\Omega)}{\|d_{\Omega}\|_{L^{\frac{p\,q}{p-q}}(\Omega)}^p},
\end{equation}
 where  
 \[
 \mathfrak{h}_p(\Omega):=\inf_{u\in C^\infty_0(\Omega)} \left\{\int_\Omega |\nabla u|^p\,dx\, :\, \int_\Omega \frac{|u|^p}{d_\Omega^{\,p}}\,dx=1\right\}.
 \]
We  recall that, when $\Omega \subsetneq \mathbb{R}^N$ is a convex bounded open set,  it is well known that 
\[
	\mathfrak{h}_p(\Omega)= \left(\frac{p-1}{p}\right)^p,
	\] (for a proof, see \cite[Theorem 11]{MMP}). The  resulting estimate  \eqref{eq:hardy2}  in the class of the convex bounded open sets is far of being sharp, as the case $q=1$ and $p=N=2$ shows, being $C_{2,1}=1>\frac 1 4$.

	Now,  the constants $\widetilde{C}_{p,q}$ defined  in \eqref{eq:makaiaperti}  satisfy the following facts:
	\begin{itemize} 
	\item when $1<p\le N$, it holds that
\begin{equation}\label{makai0}
\widetilde{C}_{p,q}=0, \qquad \mbox{ for every } 1\le q< p \mbox{ or } q=p.
\end{equation}
Indeed, if $1<p\le N$, fixed a bounded open subset $\Omega \subsetneq \mathbb{R}^N$, we remove from it a periodic array of $n$ points and we call $\Omega_n$ the open set so constructed. 
	As $p\le N$, points in $\mathbb{R}^N$ have zero $p$-capacity (see \cite[Chapter 17]{T}), then, for every $1\le q<p\le N$ or $1<q=p\le N$,  it holds
	\[ \lambda_{p,q}(\Omega_n)=\lambda_{p,q}(\Omega), \quad \mbox{ for every } n \in \mathbb{N}. \]
Since $r_{\Omega_n}$ tends to $0$, as $n \to \infty$, the above equality implies that 

\[
\widetilde{C}_{p,q}\le \limsup_{n \to \infty} \lambda_{p,q}(\Omega_n)\,\|d_{\Omega_n}\|_{L^{\frac{p\,q}{p-q}}(\Omega_n)}^p \le \lambda_{p,q}(\Omega)\, |\Omega|^{\frac{p-q}{q}} \limsup_{n \to \infty} \, r^{\,p}_{\Omega_n} =0,
\]
for every $1\le q<p\le N$ or $1<q=p\le N$. In particular, in this range,   it follows that 
\begin{equation}\label{compare}
\widetilde{C}_{p,q} < C_{p,q};
\end{equation}

\item when  $p>N$, as shown independently by  Lewis in \cite{Lewis} and Wannebo in \cite{Wan},  every open subset $\Omega\subsetneq \mathbb{R}^N$ satisfies the Hardy inequality \eqref{eq:hardy} and it  holds  
\[\mathfrak{h}_{p}(\Omega)\ge \left(\frac{p-N}{p}\right)^p,\]
(for the latter, see \cite{Av} and \cite{GPP}).
	 By using the above lower bound in  \eqref{eq:hardy2}, we get that 
	\begin{equation}\label{casop>N}\widetilde{C}_{p,q} \ge \left(\frac{p-N}{p}\right)^p>0, \qquad \mbox{for every } N<p<\infty \mbox{ and }1\le q< p, \end{equation}
	and the natural question that arises  is whether the strict inequality \eqref{compare} also  holds  in the case $p>N$,  for    $1 \le q < p$ or $q= p$. 
We  will address  Section \ref{sec:exopen}  to discuss this question and, by means of a perturbative argument, we  will be able to show that, for every $N\ge 2$ and  for every fixed $1 \le q <N$, there exists $\overline{p}=\overline{p}(q)>N$ such that  \eqref{compare} holds  for every $p\in (q, \overline{p}]$.
\end{itemize}

\vskip.2cm

\noindent Another interesting class of sets where the quantity \eqref {prodotto} is bounded from below by a positive constant is   that one of planar simply connected open sets.
 Indeed, if $p=N=2$, thanks to a result due to Ancona (see \cite{Ancona}), every simply connected open set $\Omega \subsetneq \mathbb{R}^2$ verifies the  Hardy  inequality \eqref{eq:hardy}   with the optimal Hardy constant satisfying 
  \[\mathfrak{h}_{2}(\Omega) \ge \frac{1}{16}.\] 
	Hence, for $N=2$ and every  $1\le q\le 2$, defined 
 \begin{equation}\label{eq:makaiconn}
		\widehat{C}_{2,q}=\inf \left\{ \lambda_{2,q}(\Omega)\,\|d_{\Omega}\|_{L^{\frac{2\,q}{2-q}}(\Omega)}^2:  \, \Omega\subsetneq\mathbb{R}^2 \mbox{ is a simply connected open set}, d_{\Omega}\in L^{\frac{2\,q}{2-q}}(\Omega) \right\},
		\end{equation}
the joint application of \eqref{eq:hardy2} and \eqref{makai0} implies that
\[ \widehat{C}_{2,q}\ge \frac{1}{16}>0=\widetilde{C}_{2,q}, \qquad \hbox{ for every } 1\le q\le 2.\]
 By using a different argument, in \cite{makai65} Makai shows that
\[ \frac 1 4 \le \widehat{C}_{2,2}<\frac{\pi^2} 4=C_{2,2}, \]
and, in order to prove the upper bound, he exhibits a simply connected open set $\Omega\subsetneq\mathbb{R}^2$ satisfying 
\begin{equation}\label{eq:makaiconnecset}	
\lambda_{2}(\Omega)\,r^{\,2}_{\Omega}<\frac{\pi^2} 4.
\end{equation}
In Section \ref{sec:exsc}, 	
we use this fact   to show that there exists $1 \le \overline{q}<2$ such that it holds  
\[
\widehat{C}_{2,q}< C_{2,q}, \qquad \mbox{for every } q\in  [\overline{q},  2].
\]
In addition, by exploiting \eqref{eq:makaiconnecset},  we finally  prove  that,  in the case $N=2$, there exists $\overline{p}>2$ such that  
\[0<\widetilde{C}_{p,p}<\left(\frac{\pi_p}2\right)^p={C}_{p,p}, \qquad \mbox{ for every }p\in (2,\overline{p}].\]

\vskip.2cm




\subsection{Plan of the paper}
In Section \ref{sec:polytopes} we introduce some basic properties of polytopes in $\mathbb{R}^N$ and  we extend, to any dimension $N \ge 2$, the covering argument applied by Makai in \cite{Makai}. In Section \ref{sec:makai}, we prove the main results stated in Theorem \ref{thm:makai} and in  Corollary \ref{cor:HP}. 
In the subsequent Section \ref{sec:HP}, we give an alternative proof for the Hersch--Protter inequality \eqref{eq:hpinequality1}, by means of a change of variables formula. Finally, in Section \ref{sec:exconstants}, we compare the sharp constant for the Makai inequality on convex sets with the optimal constants for other class of sets; in particular, we consider the class of general open sets  whose distance function satisfies \eqref{summdist} and that one of planar simply connected sets.

\begin{ack} The authors are deeply indebted to Lorenzo Brasco for having pointed out the open questions concerning the Makai inequality  and for the  interesting discussions and remarks on the subject. 
F.\,P.  and A.\,C.\,Z.  are grateful to the Dipartimento di Matematica e Informatica at University of Ferrara and to the  Dipartimento di Matematica at University of Pisa  for their hospitality. F.P.  has  been financially supported by the PRA-2022 project ''Geometric evolution problems and PDEs on variable domains'', University of Pisa. F.\,P. and A.\,C.\,Z. are members of the {\it Gruppo Nazionale per l'Analisi Matematica, la Probabilit\`a
e le loro Applicazioni} (GNAMPA) of the Istituto Nazionale di Alta Matematica (INdAM). 
\end{ack}

\section{Preliminaries}\label{sec:polytopes}

\noindent For every convex set $C \subset \mathbb{R}^N$, we will denote by  $\mathrm{relint}(C)$ and $\mathrm{relbd}(C)$ the {\it relative
interior} and the {\it relative boundary} of $C$, once we regard it as a subset of its affine hull. We define the dimension of $C$, and we denote it by $\dim (C)$, as the dimension of its affine hull. Conventionally, the empty set has dimension $-1$. 
\begin{defn} 
Let $C \subset \mathbb{R}^N$ be a not empty convex closed set. A convex subset $S\subseteq C$ is a face of $C$ if each segment $[x,y]\subset C$ satisfying $S\cap \mathrm{relint}([x,y])\ne\emptyset$ is contained in $S$.

\noindent We denote  by $\mathcal{F}(C)$  the set of all faces of $C$ and by $\mathcal{F}_{i}(C)$  the set of all faces of $C$ having dimension $i$, for every $0\le i \le \dim(C)-1$. The empty set and $C$ itself are faces of $C$; the other faces are called {\it proper}. Every $(\dim(C)-1)$-dimensional face of $C$ is called a {\it facet} of $C$. 
\end{defn}

We can summarise the main properties of faces in the following theorem (see \cite[Section 2.1]{S}).
\begin{thm}\label{thm:faces} 
Let $C\subset \mathbb{R}^N$ be a not empty convex bounded set. Then 
\begin{enumerate}
\item the faces of $C$ are closed;
\item if $F\ne C$ is a face of $C$, then $F\cap \mathrm{relint}(C)=\emptyset$;
\item if $G$ and $F$ are faces of $C$, then $G\cap F$ is a face of $C$;
\item if $G$ is a face of $F$ and $F$ is a face of $C$, then $G$ is a face of $C$;
\item each point $x\in C$ is contained in $\mathrm{relint}(F)$ for a unique face $F\in \mathcal{F}(C)$.
\end{enumerate}
\end{thm}

\begin{defn}
We say that $K\subset\mathbb{R}^N $ is a polytope if it is the convex hull of finitely many points of $\mathbb{R}^N$.
\end{defn}

\noindent We recall that, thanks to \cite[Theorem 1.1.11 and Theorem 2.4.7]{S}, a polytope is a compact convex set. Moreover, each proper face of $K$ is itself a polytope and is contained in some facet of $K$.
 \vskip.2cm
 
\noindent Furthermore, we recall that if  $K$ is a polytope and $0\in {\rm int}(K)$, then the following facts hold:
\begin{enumerate}
\item[(i)] the {\it polar set} $K^{\circ}=\{ x\in \mathbb{R}^N:\, \langle x,y \rangle \le 1, \mbox{ for every }\, y\in K\}$ is itself a polytope; 
\item[(ii)]  if $F$ is a face of $K$ then the conjugate set $\widehat{F}=\{ x\in K^{\circ}:\,\langle x,y \rangle = 1, \mbox{ for every }\, y\in F\}$ is itself a polytope such that
\[\dim(\widehat{F})= N-\dim(F)-1;\]
\item[(iii)] if $F,G\in \mathcal{F} (K) $ are such that $F\subset G$, then $\widehat{F}\supset \widehat{G}$;
\item[(iv)]  the application $F\mapsto \widehat{F}$ is a  bijection from $\mathcal{F} (K)$ to $\mathcal{F} (K^\circ)$.
\end{enumerate}

\vskip.2cm

In the sequel we need the following result.

\begin{lemma}\label{lemma:2facets} 
Let $K$ be a polytope and assume that $0\in {\rm int}(K)$. Then, 
\begin{enumerate}
\item if $S,S'$ are facets of $K$ such that $S\not=S'$ then  $S \cap  \mathrm{relint}(S')  =\emptyset$;
\item for every facet $S$ of $K$, if $z\in \mathrm{relbd}(S)$, then there exists another facet $\widetilde{S}\ne S$ of $K$ such that $z\in \widetilde{S}$. In particular, $z \in \mathrm{relbd}(S) \cap \mathrm{relbd}(\widetilde{S})$.
\end{enumerate}
\end{lemma}

\vskip.1cm


\begin{proof}
\begin{enumerate} 
\item Without loss of generality, we assume that $F=S\cap S'\ne\emptyset$.  Since $S\not=S'$, we have that $F$ is a proper face of $S'$. Then,  by Part (2) of Theorem \ref{thm:faces}, we have 
\[  S\cap \mathrm{relint}(S') \subset S \cap S' \cap \mathrm{relint}(S')= F \cap \mathrm{relint}(S') = \emptyset; \]
\vskip.1cm
\item let $S$ be a facet of $K$ and let $z\in  \mathrm{relbd}(S)$.  
Then, there exists a facet $F$ of $S$  such that $z\in F$. Then, $F \in \mathcal{F}(K)$ and $\dim(F)=N-2$. This implies $\dim(\widehat{F})= 1$.  
Therefore, there exist exactly two points $x',x''$, such that $x'\ne x''$ and \[\widehat{F}=[x',x''].\] 
Then, thanks to property (iv) above, there exist two facets $S',S''\in \mathcal{F}(K)$, with $S'\ne S''$, such that $\widehat{S'}=\{x' \}$ and $\widehat{S''}=\{x''\}$, and, by property (ii) above,  
\[ z\in F\subseteq \widehat{ \{x' \}} \cap \widehat{\{x'' \}}= S'\cap S'',\] namely, there exists at least another facet $\widetilde{S}\not=S$   containing $z$. 
 Since $z\in \mathrm{relbd}(S)\cap \widetilde{S}\subseteq S\cap \widetilde{S}$ and, thanks to Part (1) of this lemma, $S\cap \mathrm{relint}(\widetilde{S})=\emptyset,$  we can conclude that 
 $z \in \mathrm{relbd}(S) \cap \mathrm{relbd}(\widetilde{S})$.

\end{enumerate}
\end{proof}
In the next lemma,  we extend to every dimension  the argument applied by Makai in  \cite{Makai}  in order to divide  the interior  of a polytope $K$, when  $\mathrm{int}(K)\not=\emptyset$,  in a finite number of suitable convex open subsets. 
\begin{lemma}\label{partizione}
\noindent Let $N\ge 2$ and $K \subset \mathbb{R}^N$ be a polytope such that $0\in \mathrm{int}(K)$. Let $\Omega =\mathrm{int}{(K)}$ and let $S_1,S_2,\dots, S_h$ be the  facets of $K$. 
For  every $i \in \{1,\dots, h\}$, let $\Pi_{i}: \mathbb{R}^N\to H_i$ be the orthogonal projection on the affine hyperplane $H_i$ containing $S_i$.  Define   
\begin{equation}\label{defOm_i}
\Omega_i = \Big\{ x \in \Omega : d_{S_i}(x) = d_{\Omega}(x) \Big\},
\end{equation}
where \[d_{S_i}(x)=\min_{y\in S_i}|x-y|.\]
Then, for every $i \in \{1,\dots, h\}$, the following facts hold: 
\vskip.1cm
\begin{enumerate}
\item if $x\in \Omega_i$ then $\Pi_i(x)\in S_i$. In particular $\Pi_{i} (x)$ is the unique minimizer of the problem defining $d_{S_i}(x)$;
\vskip.1cm
	\item $\Pi_{i}(x)\in \mathrm{relint}(S_i)$, for every $x\in \Omega_i$; 
	\vskip.1cm
	\item $\Omega_i$ is a convex set; 
	\vskip.1cm
	\item 
	${\rm int}(\Omega_i)\ne\emptyset$;
	\vskip.1cm
	\item $\Omega_i$ can be included in a rectangle with base $S_i$ and height $r_{\Omega}$;
	\vskip.1cm
	\item for every $x_0 \in {\rm relint } (S_i)$, there exists a unique  $y_0\in \partial \Omega_i \cap \Omega$ such that $\Pi_i(y_0)=x_0$.
\end{enumerate} 
In particular, the restriction $\Pi_i: \partial \Omega_i\cap \Omega\to \mathrm{relint}(S_i)$ is a continuous bijection.
\end{lemma}

\begin{figure}
	\begin{tikzpicture}[thick, scale=0.8, every node/.style={scale=0.8}]	\draw[black] (-7,-1) -- (7,-1);
		\draw[black] (-7,1) -- (7,1);
		\draw[black] (-7,-1) -- (-7,1);
		\draw[black] (7,-1) -- (7,1);
		
		\draw[black, dashed] (-7,1) -- (-6,0);
		\draw[black, dashed] (-7,-1) -- (-6,0);
		\draw[black, dashed] (7,1) -- (6,0);
		\draw[black, dashed] (7,-1) -- (6,0);
		\draw[black, dashed] (-6,0) -- (6,0);
		
	\end{tikzpicture}
	\caption{A polytope $\Omega \subset \mathbb{R}^2$ divided into subsets $\Omega_i$, with $i \in \{1, \dots, 4\}$.} 		
\end{figure}
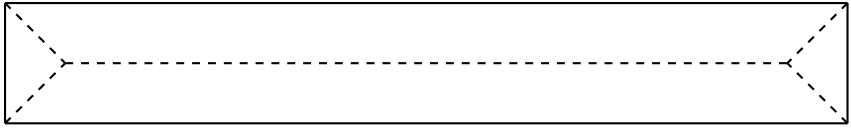



\begin{proof}
 First of all, we notice that, since $\mathrm{int}(K)\ne\emptyset$, we have that $\overline{ \Omega}=K$. Now, we prove every part separately. Let us fix $i \in \{1,\dots, h\}$.
\begin{enumerate}
\item Let $x \in \Omega_i$. By contradiction, if $\Pi_i(x)\notin K$, then the segment $[x,\Pi_i(x)]$ would intersect $\partial \Omega$ in a point $z$. Since $z\ne\Pi_i(x)$, then $z \in S_j$, with $S_j \ne S_i$. Hence, we would have that
\[ d_{\Omega}(x)\le d_{S_j}(x)\le |x-z|<|x-\Pi_i(x)|\le d_{S_i}(x)=d_{\Omega}(x), \]
which is a contradiction. 
In particular, using the fact that  $\Pi_i(x)\in S_i$, we get that 
\[d_{S_i}(x)=\min_{y\in S_i}|x-y|\le |x-\Pi_{i} (x)|=d_{H_i}(x)\le d_{S_i}(x)\]
that is, $\Pi_{i} (x)$ is the unique minimizer of the problem defining $d_{S_i}(x)$;

\vskip.2cm

\item by contradiction, let us suppose that $\Pi_i(x)\in S_i\setminus \mathrm{relint}(S_i)=\mathrm{relbd}(S_i)$. By Lemma \ref{lemma:2facets}, Part (2), there exists $S_j\ne S_i$ such that $\Pi_i(x) \in S_j$. Since 
\[ d_{S_j}(x)\le |x- \Pi_i(x)|=d_{S_i}(x)=d_{\Omega}(x)\le d_{S_j}(x),\] 
we get  that \[ d_{S_j}(x)=|x- \Pi_i(x)|.\]
Being $\Pi_{i} (x)\in S_j$, by  uniqueness,  we  would obtain $\Pi_j(x)=\Pi_i(x)\in H_i\cap H_j$. Then  $H_i \equiv H_j$, giving the contradiction $S_i=S_j$;

\vskip.2cm

\item 
by contradiction, assume that there exist  $x,y\in \Omega_i$ and $\lambda \in (0,1)$ such that $z=\lambda\,x + (1-\lambda)\,y \notin \Omega_i$. Hence $z\in \Omega_j$ for some  $j \ne i$, and 
\begin{equation}\label{eq:contraddiction} 
d_{\Omega}(z)=d_{S_j}(z)<d_{S_i}(z). 
\end{equation}
Since $\Omega$ is a convex set, then the distance function $d_{\Omega}$ is a concave function (see \cite{AU}), hence, it follows that
\begin{equation}\label{eq:dconcava}
	d_{\Omega}(z)=d_{\Omega}(\lambda\,x + (1-\lambda)\,y) \ge \lambda \,d_{\Omega}(x) + (1-\lambda)\,d_{\Omega}(y).
\end{equation}
On the other hand, by Part (1), we have that $\Pi_i(x), \Pi_i(y)\in S_i$. By linearity, we get that 
\[ \Pi_i(z)=\lambda\,\Pi_i(x) + (1-\lambda)\,\Pi_i(y) \in S_i,\] 
which implies that
\begin{equation}\label{eq:stimaproiez}
	d_{S_i}(z)=|z-\Pi_i(z)|\le\lambda|x-\Pi_i(x)| + (1-\lambda)|y-\Pi_i(y)|=\lambda\, d_{\Omega}(x) + (1-\lambda)\,d_{\Omega}(y).
\end{equation}
By combining  \eqref{eq:dconcava},  \eqref{eq:stimaproiez} and \eqref{eq:contraddiction}, we obtain a contradiction;

\vskip.2cm

\item let $x_0\in \mathrm{relint}(S_i)$, then, by Part(1) of Lemma \ref{lemma:2facets}, we have that $x_0\notin S_j$ for every $S_j \ne S_i$.
We set \mbox{$g_j(\,\cdot\,)=d_{S_j}(\,\cdot\,)-d_{S_i}(\,\cdot\,)$}, then,  it holds that 
\[g_j(x_0)=d_{S_j}(x_0)>0. 
\]
Being $g_j$ a continuous function on $\mathbb{R}^N$, there exists $B_{\rho}(x_0)$ such that $g_j>0$ on $B_{\rho}(x_0)$, for every $j \ne i$.
Hence $d_{\Omega}=d_{S_i}$ on the open set $B_{\rho}(x_0)\cap \Omega \ne \emptyset$, which implies that 
$B_{\rho}(x_0)\cap \Omega \subset \Omega_i$, giving the desired conclusion;

\vskip.2cm

\item without loss of generality, suppose that 
\[ H_i=\{ (y,x_N)\in \mathbb{R}^N:\, x_N=0 \}.\] 
Hence, thanks to Part (3), one of the following inclusion holds
\[\Omega_i\subset  \{ (y,x_N)\in \mathbb{R}^N:\, x_N\ge0 \} \qquad \mbox{ or } \qquad \Omega_i\subset  \{ (y,x_N) \in \mathbb{R}^N:\, x_N\le0 \}.\]
In both cases, since \[|x_N|=d_{S_i}(x)=d_{\Omega}(x) \le r_{\Omega}, \qquad \mbox{ for every } x \in \Omega_i,\] we obtain the claimed conclusion;

\vskip.2cm

\item let $x_0\in \mathrm{relint}(S_i)$. Since $\mathrm{ int}(\Omega_i)\ne\emptyset$, we can take an  open half-line $r_{x_0}$ with origin $x_0$ such that it is perpendicular to $S_i$ in $x_0$ and $r_{x_0}\cap \mathrm{int}(\Omega_i)\ne\emptyset$. Being $\Omega_i$  bounded set and $r_{x_0}$ a connected set, we obtain that $r_{x_0} \cap \partial \Omega_i\ne\emptyset$. Moreover, as $\partial \Omega_i=(\partial \Omega_i\cap \Omega)\cup S_i$ and $r_{x_0} \cap S_i=\emptyset$, we also have that 
\[\Sigma_i(x_0):=r_{x_0}\cap \partial \Omega_i\cap \Omega \ne\emptyset. \]
Now, we will show that $\Sigma_i(x_0)$ consists in exactly a point.
Indeed, we note that if  $y'\in \Sigma_i(x_0)$, then
\[ ]x_0,y'[= ]x_0,x]\cup [x,y'[\subseteq  \mathrm{int}({\Omega}_i), \qquad \mbox{ for every } x\in r_{x_0} \cap \mathrm{int}({\Omega}_i).\] 
In particular, if $y',y''\in  \Sigma_i(x_0)$, with $y'\ne y''$, then we would obtain the contradiction   
\[ y''\in\, ]x_0,y'[ \subset \mathrm{int}({\Omega}_i) \qquad \mbox{ or } \qquad y'\in\, ]x_0,y''[\subset \mathrm{int}({\Omega}_i).\] 
\end{enumerate}
\end{proof}


\vskip.1cm


\section{Proof of the main result}\label{sec:makai}

\noindent In proving Theorem \ref{thm:makai}, first we will restrict ourselves to consider the case when the convex open set $\Omega$ coincides with the interior of a polytope $K$. Then, the general result will follow thanks to an approximation argument by means of polytopes, valid for convex sets.  

\vskip.3cm

\noindent{\it Proof of Theorem \ref{thm:makai}.}  By following \cite{Makai}, we divide the proof in three parts: first, we prove the lower bound \eqref{eq:makai2dim} when $\Omega$ is the interior of a polytope $K$, then, by applying an approximation argument, we show that such a lower bound holds when $\Omega\subsetneq \mathbb{R}^N$ is a general convex set. Finally, we show that \eqref{eq:makai2dim} is asymptotically sharp for slab-type sequences.

\vskip.2cm	
	
\noindent {\it Part 1: Makai's inequality for a polytope.} 
Without loss of generality, let us suppose that $0\in \Omega $. Moreover, in this step we assume that $\Omega =\mathrm{int}(K)$, where $K \subset \mathbb{R}^N$ is a polytope. 
According to the notation in Lemma \ref{partizione}, we consider the subsets $\Omega_i$ given by \eqref{defOm_i}, with $i \in \{1,\dots, h\}$.

Now, we will show that, for every $i \in \{1,\dots, h\}$ and for every $u\in C^{\infty}_0(\Omega)$, it holds    
\begin{equation} \label{casoOmegai}		\int_{\Omega_i} |u|^q \, dx \le \left( \frac{2}{\pi_{p,q}} \right)^q \left( \frac{pq+p-q}{p-q} \right)^{\frac{p-q}{p}} \left( \int_{\Omega_i} d_{\Omega}^{\frac{p\,q}{p-q}} \, dx \right)^{\frac{p-q}{p}} \left( \int_{\Omega_i} |\nabla u|^p \, dx \right)^{\frac{q}{p}}.
	\end{equation} 
Indeed, let $i \in \{1,\dots, h\}$  and, up to translations and rotations, we can assume that  \[ H_i=\{ (y,t)\in \mathbb{R}^N:\, t=0 \}\] is the affine hyperplane containing $S_i$.
Then \begin{equation}\label{coincid}
	 	t=d_{S_i}(y,t)=d_{\Omega}(y,t), \qquad \mbox{ for every }(y,t)\in \Omega_i.
	 \end{equation}  

Thanks to Lemma \ref{partizione},  we have that  $\Pi_i: \overline{\partial \Omega_i\cap \Omega} \to S_i$  is a bijective and continuous function between two compact sets. 
Hence, defining $S'_i=\{ y\in \mathbb{R}^{N-1}:\, (y,0)\in S_i\}$, we obtain that there exists a continuous function $f_i: S'_i\to [0,+\infty)$ 
such that  
\begin{equation}\label{relfunz}  
(y,t)\in \overline{\partial \Omega_i\cap \Omega} \qquad \Longleftrightarrow \qquad y\in S'_i \quad \mbox{ and } \quad t=f_i(y),
\end{equation}
and it is easy to show that
\begin{equation}\label{propOm_i}
\Omega_i = \Big\{ (y,t) \in  S'_i\times \mathbb R:\, 0< t \le  f_i(y) \Big\}.
\end{equation}
Indeed, the inclusion ``$\subseteq$'' follows by using \eqref{coincid} and \eqref{relfunz}, while the converse one ``$\supseteq$'' is an application of Parts (3) and (6) of Lemma \ref{partizione}, taking into account that 
\[f_i(y)>0 \qquad \Longleftrightarrow \qquad (y,0)\in \mathrm{relint}(S_i). \]
%

Now, we recall that 
\[\left(\frac{\pi_{p,q}}{2}\right)^p=\min_{\varphi \in W^{1,p}((0,1))\setminus\{ 0\}} \left\{\frac{\displaystyle\int_0^1 |\varphi'|^p\,dt}{\left(\displaystyle\int_0^1 |\varphi|^q\,dt\right)^{\frac{p}{q}}}:\varphi(0)=0\right\}, \] 
(see \cite[Lemma A.1]{Bra2}), which implies that, for every $s>0$ 
\begin{equation}\label{poincare}
\left( \int_{0}^{s} |\varphi|^q \, dt \right)^{\frac{p}{q}} \le \left( \frac{2}{\pi_{p,q}} \right)^p s^{\frac{pq+p-q}{q}} \int_{0}^{s} |\varphi'|^p \, dt, \qquad  \mbox{for every } \varphi \in C_0^\infty((0,s]).
\end{equation}
Hence, for every $u\in C^{\infty}_0(\Omega)$ and for every $i \in \{1,\dots,h\}$, thanks to formula \eqref{propOm_i}, by using Fubini's Theorem and  \eqref{poincare} with $s=f_i(y)$, we get
\begin{equation}\label{eq:fubi}
\begin{split}
\int_{\Omega_i} |u(x)|^q \, dx &= \int_{S'_i} \int_{0}^{f_i(y)} |u(y,t)|^q \, dt\, dy \\
&\le \left( \frac{2}{\pi_{p,q}} \right)^q \int_{S'_i} f_i(y)^{\frac{pq+p-q}{p}} \left( \int_{0}^{f_i(y)} \left|\frac{\partial u}{\partial t } (y,t)\right|^p \, dt \right)^{\frac{q}{p}} \, dy \\
&\le \left( \frac{2}{\pi_{p,q}} \right)^q \left( \int_{S'_i} f_i(y)^{\frac{pq+p-q}{p-q}} \, dt \right)^{\frac{p-q}{p}} \left( \int_{S'_i} \int_{0}^{f_i(y)} \left|\frac{\partial u}{\partial t} (y,t)\right|^p \, dy \, dt \right)^{\frac{q}{p}},
\end{split}	
\end{equation}
where we also apply an H\"older's inequality in the last line. 
Taking into account \eqref{coincid}, we have that
\begin{equation}\label{eq:fdist}
	\begin{split}
		\int_{S'_i} f_i(y)^{\frac{pq+p-q}{p-q}} \, dy &= \left( \frac{pq+p-q}{p-q}\right) \int_{S'_i} \left( \int_{0}^{f_i(y)} t^{\frac{p\,q}{p-q}} \, dt \right) dy \\
		&=  \left( \frac{pq+p-q}{p-q} \right)\int_{S'_i} \left( \int_{0}^{f_i(y)} (d_{\Omega}(y,t))^{\frac{p\,q}{p-q}} \, dt\right) \, dy \\
	 &=\left( \frac{pq+p-q}{p-q} \right)\int_{\Omega_i}  d_{\Omega}^{\frac{p\,q}{p-q}}\,dx.	
\end{split}	
	\end{equation}
	
	By combining \eqref{eq:fubi} and \eqref{eq:fdist}, for every $i \in \{1,\dots,h\}$, we   obtain  \eqref{casoOmegai}.
		
		Finally, since  
	\[\Omega=\bigcup_{i=1}^h \Omega_i,\]
	by summing with respect to the index $i \in \{1,\dots,h\}$ in  \eqref{casoOmegai}, it follows that 
	\[\int_{\Omega} |u|^q \, dx \le \left( \frac{2}{\pi_{p,q}} \right)^q \left( \frac{pq+p-q}{p-q} \right)^{\frac{p-q}{p}} \sum_{i=1}^h \left(\int_{\Omega_i} d_{\Omega}^{\frac{p\,q}{p-q}} \, dx \right)^{\frac{p-q}{p}} \left( \int_{\Omega_i} |\nabla u|^{p}\, dx \right)^{\frac{q}{p}}.\]
	By  applying  H\"older's inequality 
	\[ \| a \, b \|_{\ell^1} \le \|a\|_{\ell^r} \, \|b\|_{\ell^{r'}}, \]
	with $r=p/q$, we get 
	\[ \int_{\Omega} |u|^q \, dx \le \left( \frac{2}{\pi_{p,q}} \right)^q \left( \frac{pq+p-q}{p-q} \right)^{\frac{p-q}{p}}\left( \int_{\Omega} d_{\Omega}^{\frac{p\,q}{p-q}} \, dx \right)^{\frac{p-q}{p}} \left( \int_{\Omega} |\nabla u|^{p}\, dx \right)^{\frac{q}{p}}. \]
	and by raising to the power $p/q$ on both sides, this in turns implies  
	\[ \frac{\displaystyle\int_{\Omega} |\nabla u|^{p}\, dx}{\left(\displaystyle\int_{\Omega} |u|^q \, dx\right)^\frac{p}{q}} \ge \left( \frac{\pi_{p,q}}{2} \right)^p \left( \frac{p-q}{pq+p-q} \right)^{\frac{p-q}{q}} \frac{1}{\left( \displaystyle\int_{\Omega} d_{\Omega}^{\frac{p\,q}{p-q}} \, dx \right)^{\frac{p-q}{q}}}, \qquad \mbox{for every } u\in C^{\infty}_0(\Omega). \] 
	
\noindent Taking the infimum on $C^{\infty}_0(\Omega)$ on the left-hand side, we get that for every polytope $K \subset \mathbb{R}^N$, the set $\Omega=\mathrm{int}(K)$ satisfies the lower bound \eqref{eq:makai2dim}, as desired.

\vskip.2cm

\noindent {\it Part 2: approximation argument.}
If $\Omega\subsetneq \mathbb{R}^N$ is a general convex bounded open set, thanks to \cite[Theorem 1.8.19]{S}, for every $0<\varepsilon \ll 1$, there exists a polytope $K_{\varepsilon}$ such that
\[ K_{\epsilon}\subseteq \Omega\subseteq \overline{ \Omega }\subseteq \frac{1} {1-\varepsilon}K_{\varepsilon}.\]
In particular, we have that 
\[ (1-\varepsilon)\, \Omega\subseteq \mathrm{int}(K_{\varepsilon}).\]
Then, thanks to Part (1) of the proof, we obtain that 
\[ \begin{split}
 \frac{\lambda_{p,q}(\Omega)}{(1- \varepsilon)^{p-N+N\frac{p}{q}}}=\lambda_{p,q}((1-\varepsilon)\,\Omega)\ge \lambda_{p,q}(\mathrm{int}(K_\varepsilon)) \ge \frac{C_{p,q}}{\left(\displaystyle\int_{\mathrm{int}(K_\varepsilon)} d_{\mathrm{int}(K_\varepsilon)}^{\frac{p\,q}{p-q}} \, dx \right)^{\frac{p-q}{q}}}\ge \frac{C_{p,q}}{\left(\displaystyle\int_{\Omega} d_{\Omega}^{\frac{p\,q}{p-q}} \, dx \right)^{\frac{p-q}{q}}}, 
\end{split} \]
where the last inequality follows from the fact that 
 $d_{\mathrm{int}(K_\varepsilon)}(x) \le d_{\Omega}(x)$, for every $x \in K_{\epsilon}$. 
Then, by sending $\varepsilon \to 0^+$, we get that $\Omega$ satisfies  \eqref{eq:makai2dim}.

\vskip.2cm

\noindent {\it Part 3: sharpness.} Now we will show that estimate \eqref{eq:makai2dim} is asimptotically sharp for slab-type sequences 
\[ \Omega_L = \left( -\frac L 2 ,\frac L 2  \right)^{N-1} \times (0,1),  \qquad \mbox{ with } L \ge 1. \]
\vskip.1cm

\noindent With this aim, we denote by $y=(x_1, \dots, x_{N-1})$ a point of $\mathbb{R}^{N-1}$. Let $S_0$ and $S_1$ be the facets of $\Omega_L$ contained, respectively, in the hyperplanes $\left\{(y,t)\in \mathbb{R}^N : t=0\right\}$ and  $\left\{(y,t)\in \mathbb{R}^N : t=1\right\}$, and 
we define the {\it lateral surface} $\mathcal{S}_L$ of $\Omega_L$ as the following union:
\[ \mathcal{S}_L:= \bigcup_{i=1}^{N-1} \left\{(y,t) \in \Omega_L: x_i=\pm \frac{L}{2}, \,0< t < 1\right\}. \]
Then, we define the set $\Omega_1$ as
\[ \Omega_1 = \Big\{ x \in \Omega_L: d_{\mathcal{S}_L}(x) \ge 1 \Big\}. \]
Since $|\Omega_1|=(L-2)^{N-1}$, we obtain that  
\[|\Omega_L \setminus \Omega_1|=L^{N-1}-(L-2)^{N-1} \sim  L^{N-2}\,C_N, \qquad \mbox{ as } L \to \infty,\]
which implies
\[\int_{\Omega_L \setminus \Omega_1} d_{\Omega_{L}}^{\frac{p\,q}{p-q}} \,dx \le r_{\Omega_L}^{\frac{p\,q}{p-q}} |\Omega_L \setminus \Omega_1| \sim L^{N-2}\, C_N\,\left(\frac 1 2\right)^{\frac{p\,q}{p-q}}, \qquad \mbox{ as } L \to \infty.\]

\noindent Moreover, since, for every $(y,t) \in \Omega_1$, it holds that 
\[ d_{\Omega_{L}} (y,t)= d_{S_0\cup S_1}(y,t)=\min\{t, 1-t\},\] we obtain 
\[
\int_{\Omega_1} d_{\Omega_{L}}^{\frac{p\,q}{p-q}} \,dt =\left(\int_{\left(-\frac{L}{2}+1, \frac{L}{2}-1\right)^{N-1}} \,dy\right) \left(2\,\int_{0}^{1/2}  t^{\frac{p\,q}{p-q}} \, dt\right) =L^{N-1} \,\left(\frac 1 2\right)^{\frac{p\,q}{p-q}}\,\frac{p-q}{pq+p-q}. 
\]
 
In particular, since
\[ 
\begin{split}
\int_{\Omega_L} d_{\Omega_{L}}^{\frac{p\,q}{p-q}} \, dx &= \int_{\Omega_L \setminus \Omega_1} d_{\Omega_{L}}^{\frac{p\,q}{p-q}} \,dx + \int_{\Omega_1} d_{\Omega_{L}}^{\frac{p\,q}{p-q}} \,dx,
\end{split}
\]
we have the following asymptotic behavior
\[  \int_{\Omega_L} d_{\Omega_{L}}^{\frac{p\,q}{p-q}} \, dx \sim L^{N-1} \, \left(\frac{1}{2}\right)^{\frac{p\,q}{p-q}} \,\frac{p-q}{pq+p-q}, \qquad \mbox{ as } L \to \infty. \] 

Hence, we finally obtain
\[
\frac{C_{p,q}}{ \left( \displaystyle\int_{\Omega_L} d_{\Omega_L}^{\frac{pq}{p-q}} \, dx \right)^{\frac{p-q}{q}}}=\left( \frac{\pi_{p,q}}{2} \right)^p \left( \frac{p-q}{pq+p-q} \right)^{\frac{p-q}{q}} \frac{1}{ \left( \displaystyle\int_{\Omega_L} d_{\Omega_L}^{\frac{p\,q}{p-q}} \, dx \right)^{\frac{p-q}{q}}} \sim \left( \pi_{p,q} \right)^p \frac{1}{L^{\frac{(p-q)(N-1)}{q}}}, \qquad \mbox{ as } L \to \infty. 
\]

On the other hand, by \cite[Main Theorem]{Bra1}, the following upper bound holds 
\[ \lambda_{p,q}(\Omega) < \Big( \frac{\pi_{p,q}}{2} \Big)^p \left( \frac{P(\Omega)}{|\Omega|^{1-\frac{1}{p}+\frac{1}{q}}} \right)^p \]
and it is asimptotically sharp for slab--type sequences $\Omega_L$. Finally, since 
\[ P(\Omega_L) \sim 2 \, L^{N-1} \qquad \mbox{ and } \qquad |\Omega_L| \sim  L^{N-1}, \qquad \mbox{ as } L \to \infty,\] 
we obtain
\[ \lambda_{p,q}(\Omega_L) \sim ( \pi_{p,q} )^p \frac 1 {L^{\frac{(p-q)(N-1)}{q}}}, \qquad \mbox{ as } L \to \infty. \]
The proof is over.
\qed

\vskip.2cm

As an  easy consequence of Theorem \ref{thm:makai}, we can give an alternative proof of the sharp lower bound \eqref{eq:hpinequality}.

\vskip.3cm

\noindent{\it Proof of Corollary \ref{cor:HP}.}  
We use the same notation as in Theorem \ref{thm:makai}. Thanks to Part (3) of Lemma \ref{partizione}, it holds that
\begin{equation}\label{upbound} 
f_i(y)\le r_{\Omega}, \qquad \mbox{for every } y\in S_i'.
\end{equation}
By combining \eqref{eq:fdist} and \eqref{upbound}, we obtain that
	\[ \begin{split}
		\int_{\Omega} d_{\Omega}^{\frac{p\,q}{p-q}} \, dx
		&= \sum_{i=1}^h \int_{S_i'} \int_{0}^{f_i(y)} (d_{\Omega}(y,t))^{\frac{p\,q}{p-q}} \, dt \, dy \\
		&= \frac{p-q}{pq+p-q} \sum_{i=1}^h \int_{S_i'} (f_i(y))^{\frac{p\,q}{p-q}+1} \, dy \\
		&\le \frac{p-q}{pq+p-q} r_{\Omega}^{\frac{p\,q}{p-q}} \left( \sum_{i=1}^h \int_{S_i'} f_i(y) \, dy \right) \\
		&= \frac{p-q}{pq+p-q} r_{\Omega}^{\frac{p\,q}{p-q}} \left( \sum_{i=1}^h \int_{S_i'} \int_{0}^{f_i(y)} \, dt \, dy \right) \\
		&= \frac{p-q}{pq+p-q} r_{\Omega}^{\frac{p\,q}{p-q}} |\Omega|.
	\end{split}\]
and, by applying the above estimate in \eqref{eq:makai2dim}, we get the desired lower bound \eqref{eq:hpinequality}. Finally, we obtain that this inequality is asymptotically sharp  by using  the slab-type sequences
\[ \Omega_L = \left( -\frac L 2 ,\frac L 2  \right)^{N-1} \times (0,1),  \qquad \mbox{ with } L \ge 1, \]
as it was proved in \cite[Theorem 5.7]{BPZ1}.
\qed

\begin{rem}
Let $\Omega \subsetneq \mathbb{R}^N$ be a convex bounded open set,  then, from the proof of Corollary \ref{cor:HP}, it follows that 
\[ \int_{\Omega} d_{\Omega}^{\,\alpha} \, dx \le |\Omega|\, \frac{r_{\Omega}^{\,\alpha}}{\alpha+1}, \qquad \mbox{ for every } \alpha>0. \]
Moreover, such an estimate is sharp and equality is asymptotically attained by slab--type sequences $\Omega_L$, as in Corollary \ref{cor:HP}.  
\end{rem}


\section{Alternative proof for Hersch-Protter-Kajikiya inequality for $\lambda_p$}\label{sec:HP}
We now focus on the case $1<p=q<\infty$. In this section we provide an alternative proof for the Hersch--Protter-type inequality for $\lambda_p$ in any dimension $N \ge 2$. 
More precisely, we show the following result:
\begin{thm}[Hersch-Protter-Kajikiya inequality]\label{thm:hersch-protter}
Let $1 < p < \infty$ and let $\Omega \subsetneq \mathbb{R}^N$ be a convex bounded open set. Then, the following lower bound holds
\[
\lambda_{p}(\Omega) \ge \left( \frac{\pi_p}{2} \right)^p \frac{1}{r_{\Omega}^{\,p}}.
\]
Moreover, the estimate is sharp.
\end{thm}

\vskip.2cm

With this aim, we need some preliminary results.

\subsection{Change of variables theorem} First of all, we recall the change of variables formula which follows from \cite[Theorem 7.1]{CM} when $K=\overline{ B}_1$ (see also \cite[Example 5.6]{CM}). 

\begin{thm} 
Let $\Omega \subset \mathbb{R}^N$ be a bounded open connected set of class $C^2$. Let $l:\partial \Omega \to \mathbb{R}$ be the function given by  
\begin{equation}\label{elle}
\begin{split}
l(x)= \sup\left\{d_{\Omega}(z): z \in \overline{\Omega} \mbox{ and } x \in \Pi(z)\right\},
\end{split}
\end{equation}
where 
\[ \Pi(z)=\{ x\in \partial \Omega: d_{\Omega}(z)=|x-z| \}, \] 
and let $\Phi=\partial \Omega \times \mathbb{R} \to \mathbb{R}^N$ be the map defined by 
\[\Phi(x,t)=x +t \nu(x), \qquad \mbox{for every } (x,t) \in \partial \Omega \times \mathbb{R}, \]
where, for every $x\in \partial \Omega$, $\nu(x)$ is the inward normal unit vector to $\partial \Omega$ at $x$. Then, for every $h \in L^1(\Omega)$, it holds
\begin{equation}\label{thm:crastamalusa}
\int_{\Omega} h(x) \, d x = \int_{\partial \Omega} \left( \int_{0}^{l(x)} h(\Phi(x,t)) \, \prod_{i=1}^{N-1} (1-t k_i(x)) \, d t \right) \, d \mathcal{H}^{N-1}(x). 
\end{equation}
\end{thm}

\noindent Here $k_1(x), \dots, k_{n-1}(x)$ are the principal curvatures of $\partial \Omega$ at $x$, i.\,e. the eigenvalues of the Weingarten map $W(x):T_{x} \to T_{x}$, where $T_x$ denotes the tangent space to $\partial \Omega$ at $x$.
Thanks to the $C^2$ regularity assumption on $\Omega$, it follows that $k_i$ is continuous on $\partial \Omega$ for every $i\in \{1,\dots,N-1\}$.

\begin{rem}
We note that, when $\Omega$ is a polytope, with the notation of Lemma \ref{partizione}, we have that $l(x)=|x-\Pi^{-1}_i(x)|$, for every $x\in {\rm relint}(S_i)$.
\end{rem}

\subsection{Weighted Rayleigh quotients}
 Let $w: (0,L) \to \mathbb{R}$ be a monotone non-increasing positive function, $w\not\equiv 0$. For every $1<p<\infty$, we define the following weighted Rayleigh quotient 
\[ \mu_p(w, (0,L)):=\inf_{\psi \in C_0^{\infty}((0,L])\setminus \{0\}} \left\{ \frac{ \displaystyle\int_{0}^L |\psi'(t)|^p \, w(t) \, d t}{\displaystyle\int_{0}^L |\psi(t)|^p \, w(t) \, d t} \right\}.\]

\noindent When $w\equiv 1$ on $(0,L)$, we will write $\mu_p(w, (0,L))=\mu_p(1, (0,L))$.  

\noindent With the aim to show that $\mu_p(1, (0,L))\le \mu_p(w, (0,L))$ for every monotone non-increasing positive weight $w: (0,L) \to \mathbb{R}$, first we prove that there exists a monotone non-increasing positive  minimizer for $\mu_p(1, (0,L))$.
\begin{lemma}\label{lem:monot}
Let $1<p<\infty$, then 
\[ \mu_p(1, (0,L))=L^{-p}\left(\frac{\pi_p} 2 \right)^p.\]
In particular, there exists a positive and monotone non-decreasing solution of the minimization problem  \begin{equation}\label{eq:weightedprobl}  \mu_p(1, (0,L)) = \inf_{\psi \in W^{1,p}((0,L))\setminus \{0\}, \psi(0)=0} \left\{ \frac{ \displaystyle\int_{0}^L |\psi'(t)|^p\, dt}{\displaystyle\int_{0}^L |\psi(t)|^p\, dt} \right\}.
\end{equation}
\end{lemma}

\begin{proof} First we note that,  thanks to  the density of  $C^{\infty}_0((0,L])$ in the subspace $\{\varphi \in W^{1,p}(0,L):\, \varphi (0)=0\}$, we have that \eqref{eq:weightedprobl} holds.
Moreover, following the proof of \cite[Lemma A.1]{Bra2}, we have that
\[\mu_p(1, (0,L))=L^{-p} \mu_p(1, (0,1)) = L^{-p}\left(\frac  {\pi_p} 2 \right)^p, \]
and there exists a positive symmetric function  $v\in W^{1,p}_0\left(\left(-\frac 1 2, \frac 1 2 \right)\right)$ which is monotone non-increasing on $\left(0,\frac 1 2\right)$,  
such that 
\[ \pi_p^p= \frac{ \displaystyle\int_{0}^{\frac12} |v'(t)|^p  \, d t}{\displaystyle\int_{0}^{\frac 1 2} |v(t)|^p \, d t}.\]
Then, it is easy to show that the function $\tilde v\in W^{1,p}( (0,L))$ defined by 
\[ \tilde v(t):=v\left(\frac{L-t}{2L}\right)\] 
satisfies $\tilde v(0)=0$ and it is a positive monotone non-decreasing minimizer of the problem \eqref{eq:weightedprobl}.
\end{proof}

Now we are in position to show the following minimization result.

\begin{thm}\label{thm:quotients}
Let $1 < p < \infty$ and let $w$ be a positive and monotone non-increasing function on $(0,L)$, then
\begin{equation}\label{eq:q=p}	
\mu_p(w, (0,L)) \ge \mu_p(1, (0,L)).
\end{equation}
\end{thm}

\begin{proof} 
First,  we  show that, for every  positive and monotone non-increasing weight  $w\in L^\infty((0,L))$, it holds
\begin{equation}\label{eq:q=pinter} \mu_p(1, (0,L)) \int_{0}^L |\varphi|^p w \, d t \le \int_{0}^L |\varphi'|^p \, w \, d t, \qquad \hbox{ for every }  \varphi \in C^{\infty}_0((0,L]).
\end{equation}
Indeed, let $v \in W^{1,p}((0,L))$ be a positive and  monotone non-decreasing  eigenfunction for $\mu_p(1, (0,L))$ whose existence is ensured by Lemma \ref{lem:monot}. Then $v$ is a weak solution of
\[\begin{cases} -(|v'|^{p-2} v')'=\mu_p(1, (0,L))\, v^{p-1}, \qquad \mbox{ in } (0,L),\\
 v(0)=0.\\
 \end{cases}
 \]
Moreover, fixed a mollifier $\delta \in C^\infty_0(\mathbb{R})$, given by
\[ 
\begin{split}
\delta(x):=
\begin{cases}
 \, e^{\frac{1}{|x|^2-1}}, \quad &\mbox{ if } |x|<1, \\
0, \quad &\mbox{ if } |x| \ge 1,
\end{cases}
\end{split}
\]
for $0<\varepsilon<1$, we define  
\[ \delta_\varepsilon(x) = \frac{1}{\varepsilon} \, \delta\left(\frac{x}{\varepsilon}\right) \in C^\infty(\mathbb{R}). \] 
Then $w_\varepsilon=w*\delta_\varepsilon \in C^{\infty}((\varepsilon, L-\varepsilon))$ and, as $\varepsilon \to 0$,  
$w_\varepsilon$  pointwise converges to $w$ a.\,e. on $(0,L)$.
Moreover, $w_\varepsilon' \le 0$. Indeed, for every $t, t' \in (\varepsilon, L-\varepsilon)$ such that $t>t'$, we have that
\[ 
\begin{split}
 w_\varepsilon(t)=(w * \delta_\varepsilon)(t) &= \int_0^L \delta_\varepsilon(t-y) \, w(y) \, dy =\frac{1}{\varepsilon} \int_0^L \delta\left(\frac{t-y}{\varepsilon}\right) \, w(y) \, dy \\
&= \frac{1}{\varepsilon} \int_{t-\varepsilon}^{t+\varepsilon} \delta\left(\frac{t-y}{\varepsilon}\right) \, w(y) \, dy = \int_{-1}^{1} \delta(z) \, w(t-\varepsilon z) \, dz \\
&\le \int_{-1}^{1} \delta(z) \, w(t'-\varepsilon z) \, dz = w_\varepsilon(t'),
\end{split}
\]
where in the last inequality we use that $\delta(z) > 0$, for every $z \in (-1,1)$ and $w$ is a monotone non-increasing function. 
Now, by using that $w_\varepsilon' \le 0 $ and $v'\ge 0$  a.\,e. on $(0,L)$ (thanks to Lemma \ref{lem:monot}), for every  $\varphi \in C^{\infty}_0((0,L])$, we have that 
\[
\begin{split}
\mu_p(1, (0,L)) \int_{0}^L |\varphi|^p w_\varepsilon \, d t &= \mu_p(1, (0,L)) \int_{0}^L v^{p-1} \frac{|\varphi|^p}{v^{p-1}} w_\varepsilon \, d t \\
&= \int_{0}^L |v'|^{p-2} |v'| \left( \frac{|\varphi|^p}{v^{p-1}} \,w_\varepsilon \right)' \, d t \\
&= \int_{0}^L |v'|^{p-2} v' \left( \frac{|\varphi|^p}{v^{p-1}} \right)' w_\varepsilon \, d t + \int_{0}^L |v'|^{p-2} v' \frac{|\varphi|^p}{v^{p-1}} w_\varepsilon' \, d t \\
&\le \int_{0}^L |v'|^{p-2} v' \left( \frac{|\varphi|^p}{v^{p-1}} \right)' w_\varepsilon \, d t.
\end{split}
\]
By applying Picone's inequality on the last integral (see \cite{AH}), the above inequality implies
\[ \mu_p(1, (0,L)) \int_{0}^L |\varphi|^p w_\varepsilon \, d t \le \int_{0}^L |\varphi'|^p \, w_\varepsilon \, d t. \]
Since $\|w_\varepsilon\|_{L^\infty((0,L))}\le \|w\|_{L^\infty((0,L))} $, as $\varepsilon\to0$, by using the Dominated Convergence Theorem,  we obtain that $w$ satisfies \eqref{eq:q=pinter}. 

Now, we remove the assumption that $w$ is bounded. For every $M>0$, we define $w_M:= \min\{w, M\}\in L^\infty((0,L))$. 
By applying \eqref{eq:q=pinter}, we have that 
\[\mu_p(1, (0,L)) \int_{0}^L |\varphi|^p w_M \, d t \le \int_{0}^L |\varphi'|^p \, w_M \, d t \le \int_{0}^L |\varphi'|^p \, w \, dt, \qquad \hbox{ for every }  \varphi \in C^{\infty}_0((0,L])
\]
and, sending $M\to \infty$, we get that also $w$ satisfies \eqref{eq:q=pinter}. 

Finally,   passing to the infimum on functions $\varphi \in C^{\infty}_0((0,L])$ in \eqref{eq:q=pinter}, we obtain the desired estimate
\eqref{eq:q=p}.
\end{proof}


Now, by applying the previous results, we are in a position to show Theorem \ref{thm:hersch-protter}.

\vskip.3cm

\noindent{\it Proof of Theorem \ref{thm:hersch-protter}}. We divide the proof in two parts. 

\vskip.2cm

\noindent {\it Part 1: inequality for $C^2$  convex bounded sets.}  
We first suppose that $\Omega$ is a convex bounded open set of class $C^2$. 
Let $u\in C^\infty_0(\Omega)$ then, by using formula \eqref{thm:crastamalusa}, we have that 
\begin{equation}\label{eq:CV3}
\int_{\Omega} |u(x)|^p \, dx = \int_{\partial \Omega} \left( \int_{0}^{l(x)} |u(x+ t \nu(x))|^p \, \prod_{i=1}^{N-1} (1-tk_i(x)) \, dt \right) \, d\mathcal{H}^{n-1}(x)
\end{equation} 
and
\begin{equation}\label{eq:CV4}
\int_{\Omega} |\nabla u(x)|^p \, dx = \int_{\partial \Omega} \left( \int_{0}^{l(x)} |\nabla u(x+ t \nu(x))|^p \, \prod_{i=1}^{N-1} (1-t k_i(x)) \, dt \right) \, d\mathcal{H}^{n-1}(x).
\end{equation} 

Now we fix $x \in \partial \Omega$ and let $v\in C^\infty_0((0, l(x)])$ be defined by $v(t):=u(x+t\nu(x))$ for every $t\in (0, l(x)]$. \noindent Furthermore  we introduce the weight $w_x: (0,l(x))\to \mathbb{R}$ given by 
\[ w_x(t)=\prod_{i=1}^{N-1} (1-t k_i(x))\in L^\infty(0,l(x)).\]
It is easy to verify that the weight $w_x$ is monotone non-increasing. Moreover, we note that, by \cite[Proposition 3.10, Lemma 4.1, Theorem 6.7]{CM},   $l$ is a positive and  continuous function on $\partial \Omega$. 
In particular, there exists  $z=z(x) \in \Omega$, such that $x \in \Pi(z)$ and $l(x) = d_{\Omega}(z)$. Hence 
\[  1 - t k_i(x) >1 - d_{\Omega}(z)\,k_i(x) \ge  0, \qquad \mbox{for every } t \in (0, l(x)), \]
where the last inequality follows from \cite[Lemma 5.4]{CM}.
In particular, this implies that $w_x>0$ on $(0, l(x))$. 
Being satisfied all the hypotheses of Theorem \ref{thm:quotients}, we have that
\[ \begin{split}
\int_{0}^{l(x)} |u(x+t\nu (x))|^p \,  \, w_x(t) \, dt&=  \int_{0}^{l(x)} |v(t)|^p \, w_x(t) \, dt \\ 
&\le \frac{1}{ \mu_p(w_x, (0,l(x)))} \int_{0}^{l(x)} |v'(t)|^p \, w_x(t) \, dt\\
&\le \frac{1}{ \mu_p(1, (0,l(x)))} \int_{0}^{l(x)} |v'(t)|^p \, w_x(t) \, dt.\\
\end{split} \]

\noindent Then,  applying Lemma \ref{lem:monot}  and taking into account that $l(x) \le r_\Omega$ for every $x \in \partial \Omega$, we obtain that 
\[ \begin{split}
\int_{0}^{l(x)} |u(x+t\nu (x))|^p \,  \, w_x(t) \, dt& \le \left(\frac{2}{\pi_p} \right)^p l(x)^{\,p} \int_{0}^{l(x)} |v'(t)|^p \, w_x(t) \, dt \\
&\le \left( \frac{2}{\pi_{p}} \right)^p r_{\Omega}^{\,p} \, \int_{0}^{l(x)} |\nabla u(x+t\nu(x))|^p \, w_x(t) \, d t, \qquad \hbox{ for every } x\in \partial \Omega.
\end{split} \]
By exploiting the above estimate in \eqref{eq:CV3} and then using  \eqref{eq:CV4}, we get
\[ 
\int_{\Omega} |u|^p \, d x \le \left( \frac{2}{\pi_{p}} \right)^p {r_{\Omega}^{\,p}} \left( \int_{\partial \Omega} \left( \int_{0}^{l(x)} |\nabla u(x+t\nu(x))|^p \, w_x(t) \, d t \right)\, d \mathcal{H}^{N-1}(x) \right) = \left( \frac{2}{\pi_{p}} \right)^p {r_{\Omega}^{\,p}}   \int_{\Omega} |\nabla u|^p \, d x,
\]
that is
\begin{align*}
\frac{\displaystyle\int_{\Omega} |\nabla u|^p \, dx }{\displaystyle\int_{\Omega} |u|^p \, dx} \ge \left( \frac{\pi_{p}}{2} \right)^p \frac{1}{r_{\Omega}^{\,p}}, \qquad \hbox{ for every } u \in C^\infty_0(\Omega).
\end{align*}
Taking the infimum on $C^\infty_0(\Omega)$, we obtain that  \eqref{eq:hpinequality} holds when  $\Omega$ is a convex bounded open set of class $C^2$. 

\vskip.2cm
\noindent {\it Part 2:  inequality for  convex bounded sets.}   
 We will apply an approximation argument to show the validity of \eqref{eq:hpinequality} for every convex bounded open set. Let $\Omega \subset \mathbb{R}^N$ be a convex bounded open set, then, thanks to \cite[Section 4.3]{Eggl}, there exists a sequence $\{C_k\}_{k \in \mathbb{N}} \subset \mathbb{R}^N$ of convex bounded closed sets of class $C^2$
such that 
\begin{itemize}
\item $C_{{k+1}} \subseteq C_k\subseteq C_1$, for every $k\in\mathbb{N}$;
\vskip.1cm
\item $\overline{\Omega} \subseteq C_k$ and 
\[ d_{\mathcal H}(C_k, \overline{\Omega}) = \min \left\{ \lambda \ge 0 : \overline{\Omega} \subseteq C_k + \lambda B_1,C_k \subseteq \overline{\Omega} + \lambda B_1\right\}\le \frac{1}{k}, \]
i.\,e. $\{C_k\}_{k \in \mathbb{N}}$ converges to $\overline{\Omega}$, as $k \to \infty$, in the sense of Hausdorff.
\end{itemize}

This implies that, for every $\varepsilon>0$, there exists $k_1=k_1(\varepsilon) \in \mathbb{N}$, such that 
\[ (1-\varepsilon) \,\overline{\Omega} \subseteq C_k, \qquad \mbox{ for every } k \ge k_1, \]
which leads to
\[ (1-\varepsilon) \,\Omega \subseteq \mathrm{int}(C_k), \qquad \mbox{ for every } k \ge k_1. \]
Let $r_k$ be the inradius of $\Omega_k:= \mathrm{int}(C_k)$, for every $k \in \mathbb{N}$. Then, thanks to the  monotonicity of $\lambda_p$ with respect to the inclusion of sets and by using Part (1) on $\Omega_k$, we obtain that
\begin{equation}\label{eq:approx} 
\frac{\lambda_p(\Omega)}{(1-\varepsilon)^p} \ge \lambda_p(\Omega_k) \ge \left( \frac{\pi_p}{2} \right)^p \frac{1}{r_k^{\,p}}. 
\end{equation}
Since 
\[ \Omega\subseteq \Omega_k\subseteq C_1, \qquad \mbox{ for every }  k\in\mathbb{N},\] 
we can consider the co-Hausdorff distance between $\Omega$ and $\Omega_k$, given by 
\[ d^{\mathcal H}(\Omega_k, \Omega)= d_{\mathcal H}(C_1\setminus \Omega_k, C_1\setminus \Omega),\]
and, being $\Omega$ and $\Omega_k$ convex open sets, we also have that 
\[ d^{\mathcal H}(\Omega_k, \Omega)=d_{\mathcal H}(\partial \Omega_k, \partial \Omega)=d_{\mathcal H}(C_k, \overline{\Omega})\le \frac{1}{k}.\]
Hence, as $k\to \infty$, the sequence $\{ \Omega_k\}_{k \in \mathbb{N}}$ converges to $\Omega$  in the sense of co-Hausdorff. By using \cite[Lemma 4.4]{BBP}, we have that 
\[  r_k \to r_{\Omega}, \qquad \mbox{ as } k \to \infty.\]
Hence, by sending $k \to \infty$ and then $\varepsilon \to 0$ in \eqref{eq:approx}, we finally get \eqref{eq:hpinequality}.

\vskip.2cm
 
Finally, we notice that the inequality is sharp. Indeed, the equality can be attained by different class of sets, for example by infinite slabs, as $\mathbb{R}^{N-1} \times (0,1)$, or asymptotically by the family of {\it collapsing pyramids} $C_\alpha=\mathrm{convex\,hull}\left((-1,1)^{N-1} \cup \{(0, \dots, 0, \alpha)\}\right)$, as proved in \cite[Theorem 1.2]{Bra2}. The proof is over. \qed


\section{Makai's costants for non-convex sets}\label{sec:exconstants}
%
%
 
		As pointed out in the introduction, defined $\widetilde{C}_{p,q}$ and $\widehat{C}_{2,q} $ as in \eqref{eq:makaiaperti} and \eqref{eq:makaiconn},  the natural questions that arise are whether
		\vskip.1cm
\begin{itemize}
\item $\widetilde{C}_{p,q}<{C}_{p,q}$, when $p>N$ and $1\le q<N$;
\vskip.1cm
\item $\widetilde{C}_{p,p}<{C}_{p,p}$, when $p>N$;
\vskip.1cm
\item $\widehat{C}_{2,q}<{C}_{2,q}$, when $N=2$ and $1\le q < 2$.
\end{itemize}

In this section, we give some partial answers to the questions above.


\subsection{The case  of general open sets for $1\le q<N<p$.}\label{sec:exopen}
 We  focus on the class of general open sets of $\mathbb{R}^N$ with the aim to show that,   for every  fixed $1\le q<N$,  there exists  $\overline{p}= \overline{p}(q)>N$ such that
 \begin{equation}\label{eq:confrontoCpq} 
\widetilde{C}_{p,q}<{C}_{p,q}, \qquad \mbox{ for every } p\in (q,\overline{p}].
\end{equation}

We consider the \textit{infinite fragile tower} set $\mathcal{T}\subset \mathbb{R}^N$ defined as in \cite[Theorem 5.1, (iii)]{BPZ2}. By contruction, it  satisfies the following properties:

\vskip.2cm

\begin{itemize}
\item $d_{\mathcal{T}} \in L^1(\mathcal{T}) \cap L^\infty(\mathcal{T})$; 

\vskip.2cm

\item $\lambda_{p,q}(\mathcal{T}) = 0$, for every $1 \le q<p\le N$,

\end{itemize} 
Moreover, thanks to \eqref{eq:makaiaperti}, $\lambda_{p,q}(\mathcal{T}) > 0$, for every $p>N$ and $1\le q\le p$. In order to show \eqref{eq:confrontoCpq}, it is sufficient to prove that
\[ \lim_{p \searrow N} \lambda_{p,q}(\mathcal{T}) \left( \int_{\mathcal{T}} d_{\mathcal{T}}^{\frac{p\,q}{p-q}} \, dx\right)^{\frac{p-q}{q}} < \lim_{p \searrow N} C_{p,q}.\]
With this aim, we observe that the following convergences hold
\begin{equation}\label{lim1}
\lim_{p \searrow N} \left( \int_{\mathcal{T}} d_{\mathcal{T}}^{\frac{p\,q}{p-q}} \, dx\right)^{\frac{p-q}{q}} = \left( \int_{\mathcal{T}} d_{\mathcal{T}}^{\frac{N\,q}{N-q}} \, dx\right)^{\frac{N-q}{q}} 
\end{equation}
and 
\begin{equation}\label{lim2}
 \lim_{p\searrow  N} C_{p,q} = \lim_{p \searrow N} \left( \frac{\pi_{p,q}}{2} \right)^p \left( \frac{p-q}{pq+p-q} \right)^{\frac{p-q}{q}} = C_{N,q}.
\end{equation}
The last limit follows by taking into account  that, as  computed in \cite[equation (7)]{Ta},  
\[
\pi_{p,q}=\frac{2}{q}\,\left(1+\dfrac{q}{p'}\right)^\frac{1}{q}\,\left(1+\dfrac{p'}{q}\right)^{-\frac{1}{p}}\,B\left(\frac{1}{q},\frac{1}{p'}\right),
\]
where $p'=p/(p-1)$ and $B$ is the {\it Euler Beta function}, which is continuous on $(0,+\infty)$. If we show that
\begin{equation}\label{eq:limsup}
	\limsup_{p \searrow N} \lambda_{p,q}(\mathcal{T}) \le \lambda_{N,q}(\mathcal{T}),
\end{equation}
by using \eqref{lim1}, \eqref{lim2} and \eqref{eq:limsup}, we obtain that
\[ \lim_{p \searrow N} \lambda_{p,q}(\mathcal{T}) \left( \int_{\mathcal{T}} d_{\mathcal{T}}^{\frac{p\,q}{p-q}} \, dx\right)^{\frac{p-q}{q}} \le \lambda_{N,q}(\mathcal{T}) \left( \int_{\mathcal{T}} d_{\mathcal{T}}^{\frac{N\,q}{N-q}} \, dx\right)^{\frac{N-q}{q}} =0 < C_{N,q} = \lim_{p \searrow N} C_{p,q},\]
which gives the desired conclusion.

\noindent In order to prove the claim \eqref{eq:limsup}, we note that, for every $1\le r< \infty$ and for every open set $\Omega\subseteq \mathbb{R}^N$,  it holds
\begin{equation}\label{convnorma}
	\lim_{p \searrow r} \|\nabla \varphi\|_{L^p(\Omega)}^p = \|\nabla \varphi\|_{L^r(\Omega)}^r, \qquad \mbox{ for every } \varphi \in C^{\infty}_0(\Omega).
\end{equation}
Indeed, if $\varphi \in C^\infty_0(\Omega)$, then, for every $p> r$, it follows that
\[ \int_{\Omega} |\nabla \varphi|^p \, dx = \int_{\Omega} |\nabla \varphi|^{p-r} \, |\nabla \varphi|^r \, dx \le \|\nabla \varphi\|_{L^\infty(\Omega)} ^{p-r} \int_{\Omega} |\nabla \varphi|^r \, dx,\] 
which implies 
\[\limsup_{p \searrow r}\|\nabla \varphi\|_{L^p(\Omega)}^p \le \|\nabla \varphi\|_{L^r(\Omega)}^r.\]
 On the other hand, by Fatou's Lemma, we also have that
\[ \int_{\Omega} |\nabla \varphi|^r \, dx \le \liminf_{p \searrow r} \int_{\Omega} |\nabla \varphi|^p \, dx. \]
By applying \eqref{convnorma} with $r=N$, we obtain 
\[ \limsup_{p \searrow N} \lambda_{p,q}(\mathcal{T}) \le \limsup_{p \searrow N} \frac{\displaystyle\int_{\mathcal{T}} |\nabla \varphi|^p \, dx}{ \left(\displaystyle\int_{\mathcal{T}} |\varphi|^{q} \, dx\right)^{p/{q}}}=\frac{\displaystyle\int_{\mathcal{T}} |\nabla \varphi|^N \, dx}{\left(\displaystyle\int_{\mathcal{T}} |\varphi|^{q} \, dx\right)^{N/{q}}}, \qquad \mbox{ for every } \varphi \in C^{\infty}_0(\mathcal{T}), \]
and by taking the infimum on $C^\infty_0(\mathcal{T})$, this easily implies \eqref{eq:limsup}.

\vskip.2cm

\begin{rem}
Let $1 \le q < \infty$. 
We notice that 
	\[ 
	\lim_{p \to \infty} \left( \widetilde{C}_{p,q} \right)^{\frac{1}{p}}=\lim_{p \to \infty} \left(C_{p,q}\right)^{\frac{1}{p}}=1.
	\]
Indeed, by \cite[Corollary 6.1]{BPZ2}, it holds that
\[ \lim_{p \to \infty} \frac{\pi_{p,q}}{2} = \frac{1}{2} \frac{1}{\left(\displaystyle\int_{0}^{1} (\min\{t,1-t\})^{q} \, dx\right)^{\frac{1}{q}}}=\frac{1}{(q+1)^{\frac{1}{q}}}, 
\]
which implies
\[ \lim_{p\to \infty} \left(C_{p,q}\right)^{\frac{1}{p}} = \lim_{p \to \infty} \left( \frac{\pi_{p,q}}{2} \right)^p \left( \frac{p-q}{pq+p-q} \right)^{\frac{p-q}{q}}=1. \]
Then, taking into account also \eqref{eq:makaiaperti},  it easily follows that
\[
1 \le \liminf_{p \to \infty} \left(\widetilde{C}_{p,q}\right)^{\frac{1}{p}} \le \limsup_{p \to \infty} \left(\widetilde{C}_{p,q}\right)^{\frac{1}{p}} \le \lim_{p\to \infty} \left(C_{p,q}\right)^{\frac{1}{p}} = 1. 
\]
\end{rem}

\vskip.2cm


\subsection{The case of  simply connected open sets for $p=N=2$}\label{sec:exsc}
%

\noindent In this subsection, we restrict ourselves to the case  $N=2$ and  we prove that there exist $1 \le \overline{q}<2$ and $\overline{p}> 2$ such that 
\begin{equation}\label{eq:confrontoC2q}
\widehat{C}_{2,q}< C_{2,q}, \qquad \mbox{for every } q \in [\overline{q},2],
\end{equation}
and 
\begin{equation}\label{eq:confrontoCpp}
\widetilde{C}_{p,p}< C_{p,p},\qquad  \mbox{for every } p \in [2,\overline p].
\end{equation}

First of all, we give an explicit example of a simply connected open set $\widetilde{\Omega} \subset \mathbb{R}^2$ such that 
\begin{equation}\label{eq:stimainrlambda}
	\lambda_2(\widetilde{\Omega})\,r_{\widetilde{\Omega}}^2 < C_{2,2}= \frac{\pi^2}{4}.
\end{equation}

Let $A \subset \mathbb{R}^2$ be an annulus from which we remove a segment, that is
\[ A=\left\{(x_1,x_2) \in \mathbb{R}^2 : 1 < \sqrt{x_1^2+x_2^2} < 2 \right\} \setminus \left\{(x_1,0) \in \mathbb{R}^2 : 1<x_1<2  \right\}. \]
Then it holds that $\lambda_2(A)=\pi^2$ (see \cite{Oss}, page 551).
Now, for a fixed $0<\varepsilon<1$, we consider the simply connected open set
\[ \widetilde{\Omega}= A \cup \left\{(x_1,x_2) \in \mathbb{R}^2  :  \sqrt{4-x_2^2} \le x_1<2+\varepsilon, -\varepsilon <x_2<\varepsilon \right\} \setminus \left\{(x_1,0) \in \mathbb{R}^2 : 1<x_1 <2+\epsilon \right\}. \]
\begin{figure}
	\includegraphics[scale=.4]{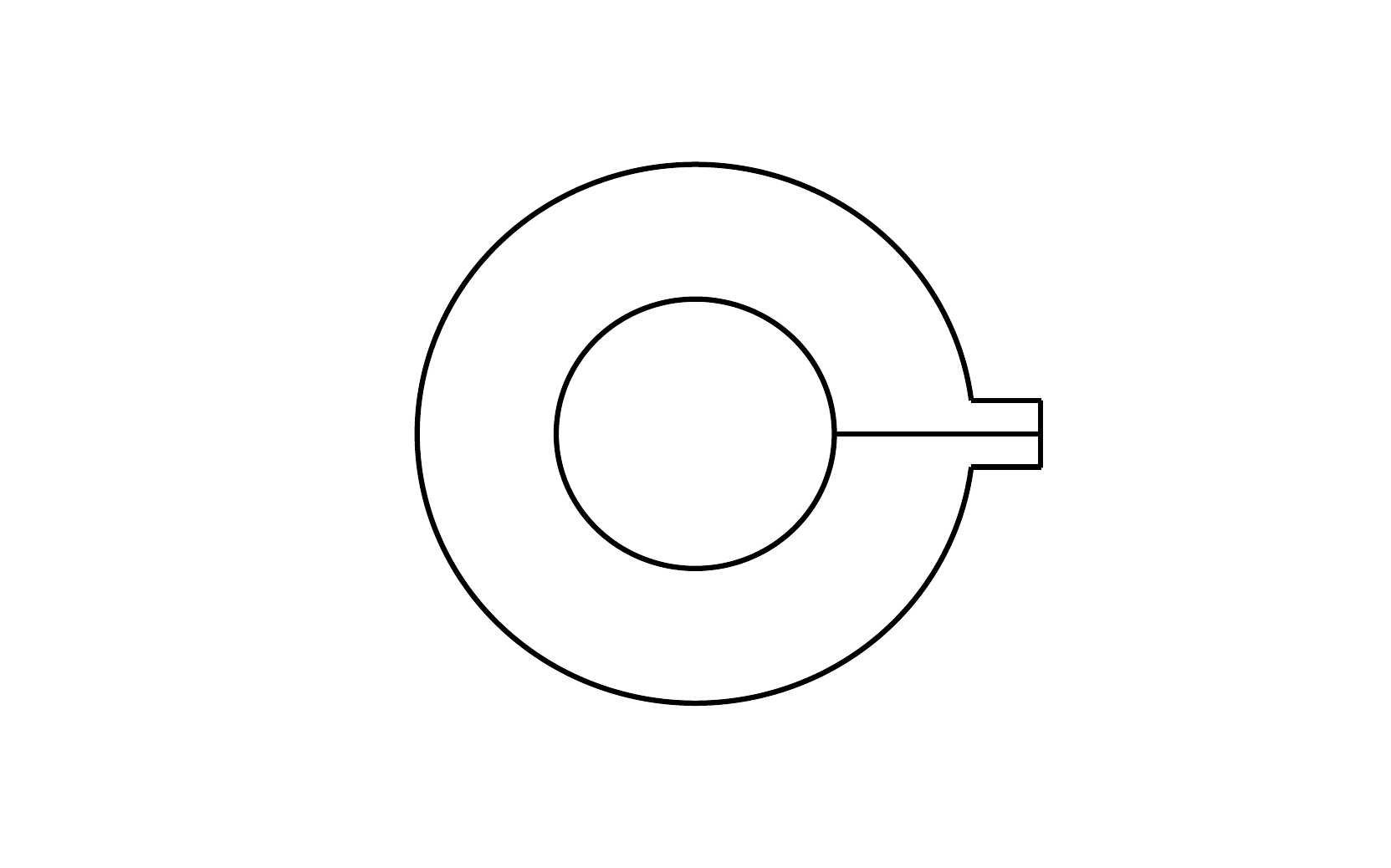}
	\caption{The set $\widetilde{\Omega}$ obtained from an annulus adding a {\it small tooth} and  removing a segment.} 
\end{figure}
Since $A \subset \widetilde{\Omega}$ and $|\widetilde{\Omega}\setminus A| \ne \emptyset$, we get that
\[ \lambda_2(\widetilde{\Omega}) <  \lambda_2(A) = \pi^2. \]
Being $r_{A}=r_{\widetilde{\Omega}}=1/2$, the above inequality implies \eqref{eq:stimainrlambda}.
Now, we recall that, by \cite[Proposition 2.3]{BPZ1}, it holds
\[\lim_{q \nearrow 2} \lambda_{2,q}(\widetilde{\Omega}) = \lambda_2(\widetilde{\Omega}) \qquad \mbox{ and } \qquad \lim_{q \nearrow 2} \pi_{2,q}=\pi, \]
hence, by combining the above limits  with \eqref{eq:stimainrlambda}, we  get 
\[ 
\begin{split}
	\lim_{q \nearrow 2} \lambda_{2,q}(\widetilde{\Omega})\,\left( \int_{\widetilde{\Omega}} d_{\widetilde{\Omega}}^{\frac{2\,q}{2-q}} \, dx\right)^{\frac{2-q}{q}} =	\lambda_2(\widetilde{\Omega})\,r_{\widetilde{\Omega}}^{\,2} < \frac{\pi^2}{4} = \lim_{q \nearrow 2} C_{2,q} .
\end{split}
\]
This gives the desired conclusion \eqref{eq:confrontoC2q}.

\noindent Finally, we note that, by applying \eqref{convnorma} with $r=2$, for every $\varphi \in C^{\infty}_0(\widetilde{\Omega})$, it holds that
\[ 
\begin{split}
\limsup_{p \searrow 2} \lambda_{p}(\widetilde{\Omega}) \le \limsup_{p \searrow 2} \frac{\displaystyle\int_{\widetilde{\Omega}} |\nabla \varphi|^p \, dx}{\displaystyle\int_{\widetilde{\Omega}} |\varphi|^{p} \, dx} 
\le  \limsup_{p \searrow 2} \frac{\displaystyle\int_{\widetilde{\Omega}} |\nabla \varphi|^p \, dx}{|\widetilde{\Omega}|^{\frac{2-p}{2}} \left(\displaystyle\int_{\widetilde{\Omega}} |\varphi|^{2} \, dx\right)^{\frac{p}{2}}} = \frac{\displaystyle\int_{\widetilde{\Omega}} |\nabla \varphi|^2 \, dx}{\displaystyle\int_{\widetilde{\Omega}} |\varphi|^{2} \, dx}, 
\end{split}
\]
and, by taking the infimum on $\varphi\in C^\infty_0(\widetilde{\Omega})$, we get that 
\[\limsup_{p \searrow 2} \lambda_{p}(\widetilde{\Omega}) \le \lambda_2(\widetilde{\Omega}).\]
Hence, 
\[ \limsup_{p \searrow 2} \lambda_{p}(\widetilde{\Omega})\, r^{\,p}_{\widetilde{\Omega}} \le \lambda_2(\widetilde{\Omega})\,r_{\widetilde{\Omega}}^{\,2}< \frac{\pi^2}{4}= \lim_{p \searrow 2} C_{p,p},\]
which implies the desired conclusion \eqref{eq:confrontoCpp}.


\end{document}